\documentclass[12pt]{article}

\usepackage[all]{xypic}
\usepackage{amsmath}
\usepackage{amsfonts}
\usepackage{amssymb}
\usepackage{amsthm}

\theoremstyle{plain}
\newtheorem{thm}{Theorem}[section]
\newtheorem{mthm}{Theorem}

\newtheorem{lmm}[thm]{Lemma}
\newtheorem{mlmm}[mthm]{Lemma}

\newtheorem{mconj}[mthm]{Conjecture}

\theoremstyle{remark}

\newtheorem{mrmk}[mthm]{Remark}

\newcommand {\tb}{\textbf}
\newcommand {\mb}{\mathbb}
\newcommand {\Z}{\mb Z}
\newcommand {\R}{\mb R}

\newcommand {\ex}{\mathrm{excess}}

\begin{document}

\title{The Hurewicz image of the $\eta_i$ family,
       a polynomial subalgebra of $H_*\Omega_0^{2^{i+1}-8+k}S^{2^i-2}$}

\author{Peter J. Eccles\footnote{\textit{peter.eccles@manchester.ac.uk}}, Hadi Zare\footnote{\textit{hzare@maths.manchester.ac.uk}}}
\date{}

\maketitle

\abstract{We consider the problem of calculating the Hurewicz image
of Mahowald's family $\eta_i\in{_2\pi_{2^i}^S}$. This allows us to
identify specific spherical classes in
$H_*\Omega_0^{2^{i+1}-8+k}S^{2^i-2}$ for $0\leqslant k\leqslant 6$.
We then identify the type of the subalgebras that these classes give rise to,
and calculate the $A$-module and $R$-module structure of these
subalgebras. We shall the discuss the relation of these calculations
to the Curtis conjecture on spherical classes in $H_*Q_0S^0$, and
relations with spherical classes in $H_*Q_0S^{-n}$.}

\tableofcontents

\section{Introduction and statement of results}

Our aim here is to detect Mahowald's $\eta_i$ using the Hurewicz
homomorphism. The $\eta_i\in{_2\pi_{2^i}^S}$ family of Mahowald was
constructed in \cite[Theorem 2]{1} as a stable composite
$$\xymatrix{S^{2^i}\ar[r]^{f_i} & X_i\ar[r]^-{g_i} & S^0}$$
with $X_i=D_{2^i-3}(\R^2,S^7)$, $i\geqslant 3$, chosen to be one of pieces in the
Snaith splitting for $\Omega^2S^9$ \cite{99}. Regardless the
construction of the complex $X_i$, it has the property that it is
highly connected such that the mapping $f_i$ can be assumed a genuine
map. Moreover, the complex $X_i$ has its top cell in a dimension
less than $2^i$ which means that the mapping $f_i$ is trivial in
homology. The mapping $g_i$ is clearly a stable mapping, and can be
realised as a genuine map after finitely many suspensions. These
together implies that the stable adjoint of the $\eta_i$ family then
may be reaslied as a mapping
$$\xymatrix{S^{2^i}\ar[r] & X_i\ar[r] & Q_0S^0}$$
where the component $f_i$ is trivial in homology. This implies that
the above composite is trivial in homology, i.e. the mapping
$\eta_i$ maps trivially under the Hurewicz homomorphism
$$h:{_2\pi_{2^i}}Q_0S^0\to H_{2^i}Q_0S^0$$
where $H_*$ denotes, and will denote, $H_*(-;\Z/2)$. Despite the
above result, one might hope that if keep desuspending the mapping
$\eta_i$, we may be able to detect it using the Hurewicz
homomorphism. This is of course is a natural thing to expect. Our
main results reads as following.\\
\textbf{Main Theorem.} \textit{Let $\eta_i\in{_2\pi_{2^i}^S}$ denote
Mahowald's family. This class is detected by the Hurewicz
homomorphism
$$h:{_2\pi_6}Q_0S^{-2^i+6}\to H_6Q_0S^{-2^i+6}.$$
The spherical class $[\eta_i]_6=h\eta_i$ has the following property.
Let $j_2^\infty:QS^{2^i-3}\to Q\Sigma^{2^i-3}P_{2^i-3}$ be the
second stable James-Hopf invariant. We then have
$$(\Omega^{2^{i+1}-9}j_2^\infty)_*[\eta_i]_6=(\Sigma^{-2^i+6}a_{2^i-3})^2\neq 0$$
where $\Sigma^{-2^i+6}a_{2^i-3}\in H_3Q_0\Sigma^{-2^i+6}P_{2^i-3}$
is the class given by the inclusion of the bottom cell $S^3\to
Q_0\Sigma^{-2^i+6}P_{2^i-3}$.}\\

We note that the space $Q_0S^{-2^i+6}$ is an infinite loop space and
it is natural to think of the subalgebra of $H_*Q_0S^{-2^i+6}$
generated by the classes of the form $Q^I[\eta_i]_6$. This problems
becomes easier to answer when we consider the unstable case and
replace infinite loop spaces with finite loop spaces. First we have
the following observation which an unstable version of our main
theorem.

\begin{mthm}
Let $\eta_i\in{_2\pi_{2^i}^S}$ denote Mahowald's family. This class
is detected by the Hurewicz homomorphism
$$h:{_2\pi_6}\Omega_0^{2^{i+1}-8}S^{2^i-2}\to H_6\Omega_0^{2^{i+1}-8}S^{2^i-2}.$$
The spherical class
$[\eta_i]_6=h\eta_i$ has the following property. Let $j_2:\Omega
S^{2^i-2}\to QS^{2^{i+1}-6}$ be the second James-Hopf invariant. We
then have
$$(\Omega^{2^{i+1}-9}j_2)_*[\eta_i]_6=g_3^2.$$
\end{mthm}

We may apply this result to detect some subalgebras living in
$H_*\Omega_0^{2^{i+1}-8}S^{2^i-2}$ and determine their algebraic
structure. In fact we are able to detect polynomial subalgebras in
$H_*\Omega_0^{2^{i+1}-8+k}S^{2^i-2}$ for $k=0,1,2,3$. Notice that
the space $\Omega_0^{2^{i+1}-8+k}S^{2^i-2}$ is a
$(2^{i+1}-8+k)$-loop space, and admits operations \cite[Part III,
Theorem 1.1]{2}
$$Q_a:H_*\Omega_0^{2^{i+1}-8+k}S^{2^i-2}\to H_{a+2*}\Omega_0^{2^{i+1}-8+k}S^{2^i-2}$$
for $a<(2^{i+1}-8)-1$. Hence, when $k=0$, we may consider to the
subalgebra of $H_*\Omega_0^{2^{i+1}-8}S^{2^i-2}$ generated by the
classes $Q_I[\eta_i]_6$ where $I=(i_1,\ldots,i_r)$ is any sequence
with $0<i_1\leqslant i_2\leqslant\cdots\leqslant i_r<2^{i+1}-9$.\\
In fact we can do more. First, notice that realising $\eta_i$ as an
element in ${_2\pi_0Q_0S^{-2^i}}$ we know that this maps
nontrivially under the Hurewicz homomorphism
$$h:{_2\pi_0QS^{-2^i}}\to H_0QS^{-2^i}.$$
Let $[\eta_i]=h\eta_i=(\eta_i)_*1$ where $1\in\overline{H}_0S^0$ is
the generator. One then may hope that this class may survive under
the homology suspension finitely many times. Second, consider the
Hurewicz homomorphism
$$h:{_2\pi_j}\Omega_0^{2^{i+1}-8+(6-j)}S^{2^i-2}\to H_j\Omega_0^{2^{i+1}-8+(6-j)}S^{2^i-2},$$
where $0\leqslant j\leqslant 6$, and let
$$[\eta_i]_j=h(\eta_i).$$
This then implies that
$$\sigma_*[\eta_i]_j=[\eta_i]_{j+1}.$$
Note that the classes $[\eta_i]_j$ are $A$-annihilated and primitive
as they are spherical. Observe that according to Theorem 1, we have
$[\eta_i]_6\neq 0$. This implies that $[\eta_i]_j\neq 0$ for $j<6$.
In particular, we have
$$[\eta_i]_5\in
H_5\Omega^{2^{i+1}-7}_0S^{2^i-2},$$
$$[\eta_i]_4\in H_4\Omega^{2^{i+1}-6}_0S^{2^i-2},$$
$$[\eta_i]_3\in H_*\Omega^{2^{i+1}-5}_0S^{2^i-2}.$$
Hence, we may consider to the subalgebra spanned by the classes of
the form $Q_I[\eta_i]_5$, $Q_I[\eta_i]_4$ and $Q_I[\eta_i]_3$ living
inside the correcponding algebras. Notice that we still don't know
the structure of this algebras, nor even if the classes
$Q_I[\eta_i]_j$ are nontrivial. Recall that having a $d$-dimensional
class $\xi$ we have $Q^{i+d}\xi=Q_i\xi$. We state our next theorems
using the operations $Q^i$ and their iterations. Our next result
gives partial information on these subalgebras.

\begin{mthm}
The homology algebra $H_*\Omega_0^{2^{i+1}-8}S^{2^i-2}$ contains a
primitively generated polynomial subalgebra given by
$$\Z/2[Q^I[\eta_i]_6:I\in\mathcal{I}_6, \ex(I)>6, i_r<2^{i+1}-3]$$
where $I=(i_1,\ldots,i_r)\in\mathcal{I}_6$ if and only if it is
admissible and all of its entries are even numbers. The action of
the Steenrod algebra on this subalgebra is determined by the Nishida
relations. Moreover, let the ideal $\underline{\mathfrak{a}}_6$ in
$H_*\Omega_0^{2^{i+1}-8}S^{2^i-2}$ be given by
$$\mathfrak{\underline{a}}_6=\langle Q^I[\eta_i]_6:\ex(I)>6, I\not\in\mathcal{I}_6\rangle.$$
Then this ideal belong to the kernel of $(\Omega^{2^{i+1}-9}j_2)_*$
where $j_2:\Omega S^{2^i-2}\to QS^{2^{i+1}-6}$ is the second
James-Hopf invariant.
\end{mthm}
In other cases, we have a similar statement.
\begin{mthm}
Let $k=1,2,3$. Then the homology algebra
$H_*\Omega_0^{2^{i+1}-8+k}S^{2^i-2}$ contains a primitively
generated polynomial subalgebra given by
$$\Z/2[Q^I[\eta_i]_{6-k}:I\in\mathcal{I}_{6-k}, \ex(I)>6-k, i_r<2^{i+1}-3],$$
with
$$\begin{array}{lll}
\mathcal{I}_5 & = &\{I:I\textrm{ admissible },Q^IQ^3\neq 0\},\\
\mathcal{I}_4 & = &\{I:I\textrm{ admissible },Q^IQ^3\neq 0,\textrm{ or }I=4J\},\\
\mathcal{I}_3 & = &\{I:I\textrm{ admissible },Q^IQ^3\neq 0,\textrm{
or }Q^IQ^2Q^1\neq 0\},
\end{array}$$
where $I=4J$ means that $I$ is an admissible sequence whose all
entries are divisible by $4$. The action of the Steenrod algebra on
this subalgebra is determined by the Nishida relations. Moreover,
let the ideal $\underline{a}_{6-k}$ in this algebra be given by
$$\mathfrak{\underline{a}}_{6-k}=\langle Q^I[\eta_i]_{6-k}:\ex(I)>6-k, I\not\in\mathcal{I}_{6-k}\rangle.$$
This ideal then belongs to the kernel of
$(\Omega^{2^{i+1}-9+k}j_2)_*$ where $j_2:\Omega S^{2^i-2}\to
QS^{2^{i+1}-6}$ is the second James-Hopf invariant.
\end{mthm}

\begin{mrmk}
The method of proving the above theorem can be applied to obtain a
set of generators for certain subalgebras of
$H_*\Omega_0^{2^{i+1}-8+k}S^{2^i-2}$ for $k=4,5,6$. However, it does
not tell anything about the algebraic structure of these
subalgebras.
\end{mrmk}

We have some comments on the above theorems. First, notice that
having $Q^I[\eta_i]_{6-k}$ with $I\not\in\mathcal{I}_{6-k}$, where
$k=0,1,2,3$, does not tell us much. It even does not tell us whether
or not if these terms are trivial. However, assuming that these
classes are nontrivial does not tell us about the subalgebras that
they generate. Second, we note that there is some indeterminacy in
determining the action of the Steenrod algebra on the stated
polynomial algebras in the following sense. If we are given a class
$Q^I[\eta_i]_{6-k}$ with $I\in\mathcal{I}_{6-k}$, then it is not
clear at all if $Sq^r_*Q^I[\eta_i]_{6-k}Q^J[\eta_i]_{6-k}$ with
$J\in\mathcal{I}_{6-k}$.\\
Finally, notice that in general, calculating the homology algebras
mentioned above will mostly depends on spectral sequence based
arguments. However, our method firstly provides some information
about a part of these algebras; and secondly gives geometric meaning
to some of its generators.\\

We note that previously, very little is known about the homology algebras
$H_*\Omega^{n+k}_0S^n$ and $H_*Q_0S^{-k}$. In fact we only have some information on
on the homology algebras
$H_*\Omega^{n+1}S^n$ and $H_*\Omega^{n+2}S^n$ \cite[Theorem 1.2, Corollary 1.3]{4}, as well as algebra $H_*Q_0S^{-1}$ and $H_*Q_0S^{-2}$
\cite[Theorem 1.1, Theorem 1.2]{3}. Very recently, we have described a part of $H_*Q_0S^{-2}$ \cite[section 5.8]{100}.
Moreover, we have observed that the $J$-homomorphisms detect infinite families of subalgebras inside $H_*Q_0S^{-n}$ \cite{101}.
It is almost certain that our theorems, Theorem 2 and Theorem 3, do not calculate the homology
algebras completely, nevertheless they shed light on some cases that
have not been known previously, as well as they provide some
knowledge about the algebraic structure of these algebras. In fact, they seem to detect a part of $H_*Q_0S^{-n}$ 
which is not detected by previous methods.
We finish by stating a conjecture, which predicts the behavior
of the class $[\eta_i]_6$ under the homology suspension. This reads
as following.\\ \\
\textbf{Conjecture.} The class $[\eta_i]_6\in H_6Q_0S^{-2^i+6}$ dies
under the homology suspension $\sigma_*:H_*Q_0S^{-2^i+6}\to
H_{*+1}Q_0S^{-2^i+7}$. Consequently, the subalgebra of
$H_*Q_0S^{-2^i+6}$ generated by $Q^I[\eta_i]_6$ belong to
$\ker\sigma_*$.\\

Finally we note that techniques to prove the above results maybe
applied in a wider generality. For instance, we may use the
classical Hopf invariant one elements to do a similar job. Notice
that the Hopf invariant one elements map nontrivially under the
Hurewicz homomorphism $h:{_2\pi_*}Q_0S^0\to H_*Q_0S^0$. We state the
following and leave the proof to reader.

\begin{mthm}
Let $i=0,1,2,3$ and consider $\nu\in{_2\pi_3^S}$, and let
$[\nu]_i\in H_iQ_0S^{-3+i}$ be the image of $\nu$ under the Hurewicz
homomorphism
$$h:{_2\pi_*}Q_0S^{-3+i}\to H_*Q_0S^{-3+i}.$$
This class pulls back to a spherical class $[\nu]_i\in
H_i\Omega^{7-i}S^4$. This class gives rise to a primitively
generated polynomial algebra inside $H_*\Omega^{7-i}S^4$ given by
$$\Z/2[Q^I[\nu]_i:I\in\textrm{ admissible },\ex(I)>0,i_r<6].$$
The action of the Steenrod algebra on this polynomial algebra is
completely determined by the Nishida relations. Moreover, this
subalgebra maps monomorphically under $(\Omega^{6-i}j_2)_*$ where
$j_2:\Omega S^4\to QS^6$ is the second James-Hopf invariant.
\end{mthm}

The key fact in the proof will be that $\nu$ maps to the identity
element under $j_2:{_2\pi_6}S^6\to{_2\pi_6}QS^6$. A similar
statement can be made about $\sigma\in{_2\pi_7^S}$, and the outcome
seems more interesting as we get more loops!

\begin{mthm}
Let $i=0,1,2,\ldots,7$ and consider $\sigma\in{_2\pi_7^S}$, and let
$[\sigma]_i\in H_iQ_0S^{-7+i}$ be the image of $\nu$ under the
Hurewicz homomorphism
$$h:{_2\pi_*}Q_0S^{-7+i}\to H_*Q_0S^{-7+i}.$$
This class pulls back to a spherical class $[\sigma]_i\in
H_i\Omega^{15-i}S^4$. This class gives rise to a primitively
generated polynomial algebra inside $H_*\Omega^{15-i}S^4$ given by
$$\Z/2[Q^I[\sigma]_i:I\in\textrm{ admissible },\ex(I)>0,i_r<14].$$
The action of the Steenrod algebra on this polynomial algebra is
completely determined by the Nishida relations. Moreover, this
subalgebra maps monomorphically  under $(\Omega^{14-i}j_2)_*$ where
$j_2:\Omega S^4\to QS^6$ is the second James-Hopf invariant.
\end{mthm}

In a recent work \cite{101} we have done a similar job on those elements
of ${_2\pi_*^S}$ which belong to the image of the $J$-homomorphism.
Although $\eta_i$ does not belong to the this image, however we will
use similar techniques there as well.\\

\tb{\textit{Important Note}.} 
We have detected polynomial subalgebras inside the homology algebras $H_*\Omega_0^{2^{i+1}-8+k}S^{2^i-2}$ for $k=4,5,6$. 
The application of the Steenrod operations then will detect infinitely many other terms inside these algebras that give rise to
polynomial subalgebras as done in \cite[Lemma 6.3]{3}. If we have a class $Q^I[\eta_i]_j$ such as given by previous theorems, and a class
$\xi$ such that $Sq^r_*\xi=Q^I[\eta_i]_j$, then we know that $\xi\neq 0$. The relations such as
$$Sq^{2r}_*\xi^2=(Sq^r_*\xi)^2=(Q^I[\eta_i]_j)^2\neq 0$$
show that $\xi^{2^t}\neq 0$. Therefore, the class $\xi$, as well as classes of the form $Q^i\xi$ for 
suitable choices of $I$, will give rise to polynomial subalgebra inside $H_*\Omega_0^{2^{i+1}-8+k}S^{2^i-2}$ for $k=4,5,6$.\\

The rest of this paper is devoted to the proof of this results and
related calculation. We start by proving our results in the unstable
case. We shall then provide the reader with the proof of our main
result.\\

\tb{Acknowledgement.} The second named author has been a self-funded visitor at the University of Manchester, and wishes
to express his gratitude towards the School of Mathematics for providing him with the facilities to carry on with this work, as well
as support of many individuals within the school. He is also grateful to his
family for the financial support. Both authors are grateful to Fred Cohen for helpful comments.

\section{Proof of Theorem 1}

The proof of our theorems are based on two basic observations. The
first observation is an equivalence between two definitions of the
Hopf invariant. We recall the following result \cite[Proposition
4.4]{5}.

\begin{lmm}
Let $\alpha\in{_2\pi_{2m}}QX$. Then $h\alpha=x_m^2$ with $x\in H_mX$
if and only if the stable adjoint of $\alpha$ is detected by
$Sq^{m+1}$ on $x_m$ in its stable mapping cone. Here $h$ is the
Hurewicz homomorphism
$$h:{_2\pi_*}QX\to H_*QX.$$
\end{lmm}

The second observation is provided by the fact that the class
$\eta_i\in{_2\pi_{2^i}^S}$ pulls back to ${_2\pi_{2^{i+1}-2}}S^{2^i-2}$, i.e.
$$S^{2^{i+1}-2}\to S^{2^i-2},$$
and maps to $\nu\in{_2\pi_3^S}$ under the second James-Hopf
invariant \cite{1}. Here by the $2$nd James-Hopf invariant we mean
$$j_2:\Omega\Sigma X\to Q(X\wedge X),$$
where in our case $X=S^{2^i-3}$. In this case, the fact that
$\eta_i$ has Hopf invariant $\nu$ means that $j_2\eta_i=\nu$. This
implies that as an unstable mapping $\nu$ is given by the following
composite
$$\xymatrix{S^{2^{i+1}-3}\ar[r]^-{\eta_i} & \Omega S^{2^i-2}\ar[r]^-{j_2} & QS^{2^{i+1}-6}.}$$
Here $\eta_i:S^{2^{i+1}-3}\to\Omega S^{2^i-2}$ is the adjoint to the
mapping $S^{2^{i+1}-2}\to S^{2^i-2}$. The mapping $\nu$ is detected
by $Sq^4$ on $g_{2^{i+1}-6}$ in its mapping cone, where
$g_{2^{i+1}-6}\in H_{2^{i+1}-6}QS^{2^{i+1}-6}$ is the generator given by the inclusion $S^{2^{i+1}-6}\to QS^{2^{i+1}-6}$.
We may adjoint down the above composite to obtain the following
composite
$$\xymatrix{\nu:S^{7}\ar[r]^-{\eta_i} & \Omega_0^{2^{i+1}-9}S^{2^i-2}\ar[rr]^-{\Omega^{2^{i+1}-10}j_2} & &QS^4.}$$
This composite it detected by $Sq^4$ on $g_4\in H_4QS^4$ in its
mapping cone. Applying Lemma 2.1 implies that if adjoint down once
more, we then obtain a mapping which is detected by homology. More
precisely, the composite
\begin{equation}
\xymatrix{\widetilde{\nu}_6:S^6\ar[r]^-{\eta_i} &
\Omega^{2^{i+1}-8}_0S^{2^i-2}\ar[rr]^-{\Omega^{2^{i+1}-9}j_2} &
&QS^3}
\end{equation}
is detected by
$$\widetilde{\nu}_{6*}g_6=g_3^2$$
where $\widetilde{\nu}_6$ denotes adjoint of $\nu$, and $g_3\in
H_3QS^3$ is a generator given by the inclusion $S^3\to QS^3$.
Setting $[\eta_i]_6=h\eta_i$ we then have $[\eta_i]_6\neq 0$ in
$H_6\Omega^{2^{i+1}-8}_0S^{2^i-2}$ and that
$$({\Omega^{2^{i+1}-7}j_2})_*[\eta_i]_6=g_3^2$$
where $h:{_2\pi_6}\Omega_0^{2^{i+1}-8}S^{2^i-2}\to
H_6\Omega_0^{2^{i+1}-8}S^{2^i-2}$ denotes the Hurewicz
homomorphism.\\

As we mentioned earlier, the homology of the space
$\Omega^{2^{i+1}-8}_0S^{2^i-6}$ admits operations
$$Q_a:H_*\Omega_0^{2^{i+1}-8+k}S^{2^i-2}\to H_{a+2*}\Omega_0^{2^{i+1}-8+k}S^{2^i-2}$$
for $a<(2^{i+1}-8)-1$. We like to investigate the $R$-module spanned
by $[\eta_i]_6$, i.e. the module spanned by elements of the form
$Q_I[\eta_i]_6$ with $I=(i_1,\ldots,i_r)$ such that
$$0<i_1\leqslant i_2\leqslant\cdots\leqslant i_r<2^{i+1}-9.$$
The mapping
$$\Omega^{2^{i+1}-9}j_2:\Omega^{2^{i+1}-8}_0S^{2^i-2}\to QS^6$$
is a $(2^{i+1}-9)$-fold loop map. This implies that
$(\Omega^{2^{i+1}-9}j_2)_*$ commutes with all classes of the form
$Q_I[\eta_i]_6$ with $i_r<(2^{i+1}-9)-1=2^{i+1}-9$, i.e. having
$Q_I[\eta_i]_6$ with $0<i_1\leqslant\cdots\leqslant i_r<2^{i+1}-9$
then we have
$$(\Omega^{2^{i+1}-9}j_2)_*Q_I[\eta_i]_6=Q_I(\Omega^{2^{i+1}-9}j_2)_*[\eta_i]_6=Q_Ig_3^2.$$
Let us write $I=2K$ if $K=(k_1,\ldots,k_r)$ with $i_j=2k_j$ for any
$1\leqslant j\leqslant r$. We then have
$$\begin{array}{lllll}
(\Omega^{2^{i+1}-9}j_2)_*Q_I[\eta_i]_6 & = &\left\{\begin{array}{ll} 0          &\textrm{ if }i_j\textrm{ is odd for some }j\\ \\
                                                                    (Q_Kg_3)^2  &\textrm{ if }I=2K.
\end{array}\right.
\end{array}$$

This implies that if $I=2K$, then $Q_I[\eta_i]_6\neq 0$. On the
other hand notice that $\Omega^{2^{i+1}-7}j_2$ is an iterated loop
map, which in particular implies $(\Omega^{2^{i+1}-7}j_2)_*$ is a
multiplicative map. Also, notice that $H_*QS^3$ is a polynomial
algebra. Hence, if we have an arbitrary pair of terms
$Q_I[\eta_i]_6$, $Q_L[\eta_i]_6$ which map nontrivially under
$(\Omega^{2^{i+1}-7}j_2)_*$ then their product will map nontrivially
under this homomorphism. This then implies that
$$\Z/2[Q_I[\eta_i]_6: I=2K\textrm{ increasing, }i_1>0,i_r<2^{i+1}-9]$$
is a polynomial algebra living in
$H_*\Omega_0^{2^{i+1}-8}S^{2^i-2}$. Recall that for a
$d$-dimensional class $\xi$ we have $Q_a\xi=Q^{a+d}\xi$. Hence, we
may rewrite the above polynomial algebra as
$$\Z/2[Q^I[\eta_i]_6: I=2K\textrm{ admissible }i_1>0,i_r<2^{i+1}-3].$$
We note that $I=2K$ are all of the sequences living in
$\mathcal{I}_6$. Finally notice that the class $[\eta_i]_6$ is an
$A$-annihilated class. Hence, to describe the action of Steenrod
operations $Sq^t_*$ on $Q^I[\eta_i]_6$ we only need to apply Nishida
relations. This completes the proof of Theorem 1.

\section{Proof of Theorem 2}

The proof of this result is similar to the proof of Theorem 1. We
like to draw reader's attention to the following table, where the
left hand side denotes the mapping $\nu$, suspended down, and the
right hand side denotes the Hurewicz image of the corresponding
mapping
$$\begin{array}{lllll}
\widetilde{\nu}_6:S^6\to QS^3    & & &h\widetilde{\nu}_6=Q^3g_3,\\
\widetilde{\nu}_5:S^5\to QS^2    & & &h\widetilde{\nu}_5=Q^3g_2,\\
\widetilde{\nu}_4:S^4\to QS^1    & & &h\widetilde{\nu}_4=Q^3g_1+Q^2Q^1g_1,\\
\widetilde{\nu}_3:S^3\to Q_0S^0    & & &h\widetilde{\nu}_3=x_3+Q^2x_1+D,\\
\widetilde{\nu}_2:S^2\to Q_0S^{-1} & & &h\widetilde{\nu}_2=w'_2,\\
\widetilde{\nu}_1:S^1\to Q_0S^{-2} & & &h\widetilde{\nu}_1=p_1^{S^{-2}},\\
\end{array}$$
where $D$ denotes a sum of decomposable terms, $w'_2\in
H_2Q_0S^{-1}$ is an $A$-annihilated primitive class with
$\sigma_*w'_2=p'_3=x_3+Q^2x_1+D$, and $p_1^{S^{-2}}\in H_1Q_0S^{-2}$
is an $A$-annihilated primitive class with
$\sigma_*p_1^{S^{-2}}=w'_2$. Recall that $(1)$ provided us with a
decomposition for $\widetilde{\nu}_6$. This allows us to have the
following decompositions for $\widetilde{\nu}_5$,
$\widetilde{\nu}_4$, and $\widetilde{\nu}_3$ respectively
$$\xymatrix{
\widetilde{\nu}_5:S^5\ar[r]^-{\eta_i} & \Omega^{2^{i+1}-7}_0S^{2^i-2}\ar[rr]^-{\Omega^{2^{i+1}-8}j_2} & &QS^2,\\
\widetilde{\nu}_4:S^4\ar[r]^-{\eta_i} & \Omega^{2^{i+1}-6}_0S^{2^i-2}\ar[rr]^-{\Omega^{2^{i+1}-7}j_2} & &QS^1,\\
\widetilde{\nu}_3:S^3\ar[r]^-{\eta_i} &
\Omega^{2^{i+1}-5}_0S^{2^i-2}\ar[rr]^-{\Omega^{2^{i+1}-6}j_2} &
&Q_0S^0.}$$ Now we can complete proof of Theorem 2. We only do one
case and leave the other cases to the reader.\\
Consider $\widetilde{\nu}_4:S^4\to QS^1$ with
$h\widetilde{\nu}_4=Q^3g_1+Q^2Q^1g_1$. This then implies that
$$[\eta_i]_4=h\eta_i\neq 0.$$
Moreover, this shows that
$$(\Omega^{2^{i+1}-7}j_2)_*[\eta_i]_4=Q^3g_1+Q^2Q^1g_1.$$
Next, we like to consider the subalgebra of generated
$H_*\Omega^{2^{i+1}-6}S^{2^i-2}$ by classes $Q^I[\eta_i]_4$. The
homology of the space $\Omega^{2^{i+1}-6}S^{2^i-2}$ admits
operations
$$Q_a:H_*\Omega^{2^{i+1}-6}_0S^{2^i-2}\to H_{a+2*}\Omega^{2^{i+1}-6}_0S^{2^i-2}$$
with $a<(2^{i+1}-6)-1$. This then implies that the mapping
$(\Omega^{2^{i+1}-7}j_2)_*$ commutes with $Q_I[\eta_i]_4$ where
$I=(i_1,\ldots,i_r)$ such that $0<i_1\leqslant\cdots\leqslant
i_r<2^{i+1}-7$. Notice that written with operations $Q^I$ we then
look for the subalgebra generated by the classes of the form
$Q^I[\eta_i]_4$ with $I$ admissible and $i_r<2^{i+1}-3$. This yields
the following
$$\begin{array}{lll}
(\Omega^{2^{i+1}-7}j_2)_*Q^I[\eta_i]_4 & = & Q^I(\Omega^{2^{i+1}-7}j_2)_*[\eta_i]_4\\
                                       & = & Q^I(Q^3g_1+Q^2Q^1g_1)\\
                                       & = & Q^IQ^3g_1+Q^IQ^2Q^1g_1.
\end{array}$$
Notice that in the above sum the second term is of the form
$Q^Ig_1^4$. Therefore, the above sum is nontrivial only if either
$Q^IQ^3\neq 0$, or all entries of $I$ are divisible by $4$. Notice
that this characterises the set of sequences belonging to
$\mathcal{I}_4$. The fact that $H_*QS^1$ is a polynomial algebra,
combined with the fact that $(\Omega^{2^{i+1}-7}j_2)_*$ is a
multiplicative map, implies that the subalgebra of
$H_*\Omega^{2^{i+1}-6}_0S^{2^i-2}$ generated by classes of the form
$Q^I[\eta_i]_4$ is a polynomial algebra, i.e. we have a primitively
generated subalgebra sitting inside $H_*\Omega^{2^{i+1}-6}S^{2^i-2}$
determined by
$$\Z/2[Q^I[\eta_i]_4:I\in\mathcal{I}_4, \ex(I)>4, i_r<2^{i+1}-3].$$
Notice that if $I\not\in\mathcal{I}_4$ then
$(\Omega^{2^{i+1}-7}j_2)_*Q^I[\eta_i]_4=0$. This means that the
ideal $\mathfrak{\underline{a}}_4\subseteq
H_*\Omega^{2^{i+1}-6}_0S^{2^i-2}$ generated by such classes belongs
to the kernel of $(\Omega^{2^{i+1}-7}j_2)_*$, i.e.
$$\mathfrak{\underline{a}}_4=\langle Q^I[\eta_i]_4:\ex(I)>4, I\not\in\mathcal{I}_4\rangle\subseteq\ker(\Omega^{2^{i+1}-7}j_2)_*.$$
This completes the proof of Theorem 2.

\section{Stablisation: The Main Theorem}
We like to restate our results when the finite loop spaces are
replaced with infinite loop spaces. More precisely, notice that
there is a mapping
$$E:\Omega S^{2^i-2}\to QS^{2^i-3}.$$
Applying the iterated loop functor $\Omega^{2^{i+1}-9}$ to this
mapping we obtain
$$\Omega^{2^{i+1}-9}E:\Omega^{2^{i+1}-8}S^{2^i-2}\to QS^{-2^i+6}$$
where restricting to base point components yields
$$\Omega_0^{2^{i+1}-8}S^{2^i-2}\to Q_0S^{-2^i+6}.$$
We then may consider the mapping
$$(\Omega^{2^{i+1}-9}E)_*:H_*\Omega^{2^{i+1}-8}_0S^{2^i-2}\to
H_*Q_0S^{-2^i+6}$$ and the image of the polynomials identified by
Theorem 2.\\
Previously, we used James-Hopf invariant $j_2:\Omega S^{2^i-2}\to
QS^{2^{i+1}-6}$ and its iterated loop. In the stable case, we
consider the stable James-Hopf invariant
$$j^\infty_2:QS^{2^i-3}\to Q\Sigma^{2^i-3}P_{2^i-3}$$
where the upper index $\infty$ is used to note that this is a map
associated with infinite loop spaces. Applying $\Omega^{2^{i+1}-9}$
to $j^\infty_2$ we obtain
$$Q_0S^{-2^i+6}\to Q_0\Sigma^{-2^i+6}P_{2^i-3}.$$
We recall that there is a commutative diagram given by
\begin{equation}
\xymatrix{\Omega S^{2^i-2}\ar[r]^-{j_2}\ar[d]_-E & QS^{2^{i+1}-6}\ar[d]^-i\\
            QS^{2^i-3}\ar[r]^-{j^\infty_2}    &
            Q\Sigma^{2^i-3}P_{2^i-3}.}
            \end{equation}
In particular, the mapping $S^{2^{i+1}-6}\to QS^{2^{i+1}-6}\to
Q\Sigma^{2^i-3}P_{2^i-3}$ may be viewed as the inclusion of the
bottom cell, and is nontrivial in homology. Applying
$\Omega^{2^{i+1}-9}$ to this diagram we obtain
$$\xymatrix{\Omega_0^{2^{i+1}-8}S^{2^i-2}\ar[r]\ar[d] & QS^3\ar[d]\\
            Q_0S^{-2^i+6}\ar[r]                       &
            Q_0\Sigma^{-2^i+6}P_{2^i-3}.}$$
Notice that we like to study the composite
$$j_2^\infty\eta_i:S^{2^{i+1}-3}\to QS^{2^i-3}\to Q\Sigma^{2^i-3}P_{2^i-3}.$$
This allows us to restrict our attention to
\begin{equation}
j_2^\infty\eta_i:S^{2^{i+1}-3}\to QS^{2^i-3}\to
Q\Sigma^{2^i-3}P^{2^i}_{2^i-3}.
\end{equation}
The fact that $\eta_i$ maps to $\nu\in{_2\pi_3^S}$ under the Hopf
invariant implies that (3) should be detected by $Sq^4$ on the
bottom cell, i.e. by $Sq^4$ on $\Sigma^{2^i-3}a_{2^i-3}$. Like the
proof of Theorem 2, adjointing down, $(2^{i+1}-10)$-times, we obtain
\begin{equation}
\xymatrix{S^7\ar[r]&
QS^{-2^i+7}\ar[rr]^-{\Omega^{2^{i+1}-10}j_2^\infty}& &
Q\Sigma^{-2^i+7}P_{2^i-3}}.
\end{equation}
Our claim then is that this mapping is detected by $Sq^4$ on a
$4$-dimensional homology class, say $\Sigma^{-2^i+7}a_{2^i-3}\in
H_4Q\Sigma^{-2^i+7}P_{2^i-3}$. It is not difficult to see that there
is a such homology class. Applying the iteration loop functor
$\Omega^{2^{i+1}-10}$ to diagram (2) and taking homology results the
following commutative diagram
$$\xymatrix{
H_*QS^{2^{i+1}-6}\ar[rr]^-{i_*}&                      & H_*Q\Sigma^{2^i-3}P_{2^i-3}\\
H_*QS^4\ar[u]^-{\sigma_*^{2^{i+1}-10}}\ar[rr]^-{(\Omega^{2^{i+1}-10}i)_*}&
& H_*Q\Sigma^{-2^i+7}P_{2^i-3}.\ar[u]_-{\sigma_*^{2^{i+1}-10}}}$$
Here we have used $\sigma_*^{2^{i+1}-10}$ to denote the iterated
homology suspension. Notice that
$$\sigma_*^{2^{i+1}-10}(\Omega^{2^{i+1}-10}i)_*g_4=i_*\sigma_*^{2^{i+1}-10}g_4=i_*g_{2^{i+1}-6}=\Sigma^{2^i-3}a_{2^i-3}.$$
This allows us to define
$$\Sigma^{-2^i+7}a_{2^i-3}=(\Omega^{2^{i+1}-10}i)_*g_4$$
with the property that
$$\sigma_*^{2^{i+1}-10}\Sigma^{-2^i+7}a_{2^i-3}=\Sigma^{2^i-3}a_{2^i-3}.$$
Here $g_4\in H_4QS^4$ is the generator given by $S^4\to QS^4$.
Similarly, we may define $\Sigma^{-2^i+6}a_{2^i-3}\in
H_3Q\Sigma^{-2^i+6}P_{2^i-3}$ by
$$\Sigma^{-2^i+6}a_{2^i-3}=(\Omega^{2^{i+1}-9}i)_*g_3.$$
The observation that $g_3\in H_3QS^3$ and $g_4\in H_4QS^4$ are
spherical implies that the classes $\Sigma^{-2^i+6}a_{2^i-3}$,
$\Sigma^{-2^i+7}a_{2^i-3}$ are also spherical classes in the
respective homology groups. Notice that these are quite natural to
expect, as for instance $\Sigma^{-2^i+6}a_{2^i-3}$ corresponds to
the bottom call of $\Sigma^{-2^i+6}P_{2^i-3}$ whereas we know that a
bottom cells always are given by spherical classes.\\

Now we ready to prove our Main Theorem. We recall the statement that
we want to prove.\\
\textbf{Main Theorem.} \textit{Let $\eta_i\in{_2\pi_{2^i}^S}$ denote
Mahowald's family. This class is detected by the Hurewicz
homomorphism
$$h:{_2\pi_6}Q_0S^{-2^i+6}\to H_6Q_0S^{-2^i+6}.$$
The spherical class $[\eta_i]_6=h\eta_i$ has the following property.
Let $j_2^\infty:QS^{2^i-3}\to Q\Sigma^{2^i-3}P_{2^i-3}$ be the
second stable James-Hopf invariant. We then have
$$(\Omega^{2^{i+1}-9}j_2^\infty)_*[\eta_i]_6=(\Sigma^{-2^i+6}a_{2^i-3})^2\neq 0$$
where $\Sigma^{-2^i+6}a_{2^i-3}\in H_3Q_0\Sigma^{-2^i+6}P_{2^i-3}$
is the class given by the inclusion of the bottom cell $S^3\to
Q_0\Sigma^{-2^i+6}P_{2^i-3}$.}\\

Here we use $[\eta_i]_6$ to denote this spherical class as we like
to remember that it is the class given by the mapping
$$\Omega_0^{2^{i+1}-8}S^{2^i-2}\to Q_0S^{-2^i+6}.$$
To complete the proof, we need a more general version of Lemma 6.
The result is as following.
\begin{lmm}
Suppose $f:S^{2m}\to\Omega X$ is given with $X$ having its bottom
call in dimension $m+1$. Then the adjoint mapping $S^{2m+1}\to X$ is
detected by $Sq^{m+1}$ on $\sigma_*x_m$ if and only if $hf=x_m^2\neq
0$ where $x_m\in H_*\Omega X$.
\end{lmm}

We leave the proof of this lemma to another section.\\ \\
\textit{Proof of the Main Theorem .} We have already done a part of
the proof above. To complete the proof, notice that the composite
$$\xymatrix{S^7\ar[r]& QS^{-2^i+7}\ar[rr]^-{\Omega^{2^{i+1}-10}j_2^\infty}& & Q\Sigma^{-2^i+7}P^{2^i}_{2^i-3}}$$
is detected by $Sq^4$ on a $4$-dimensional homology class, say
$\Sigma^{-2^i+7}a_{2^i-3}\in H_4Q\Sigma^{-2^i+7}P_{2^i-3}$.
Moreover, we know that
$$\Sigma^{-2^i+7}a_{2^i-3}=\sigma_*\Sigma^{-2^i+6}a_{2^i-3}.$$
This then implies that adjointing down once, we have
$$\xymatrix{S^6\ar[r]^{\eta_i}& QS^{-2^i+6}\ar[rr]^-{\Omega^{2^{i+1}-9}j_2^\infty}& & Q\Sigma^{-2^i+6}P^{2^i}_{2^i-3}}.$$
Lemma 7 now implies that the above composite is detected by
$$(\Omega^{2^{i+1}-9}j_2^\infty\eta_i)_*g_3=(\Sigma^{-2^i+6}a_{2^i-3})^2\neq 0.$$
This completes the proof.$\phantom{MMMMMMMMMMMMMMMMM}\Box$\\

According to the above proof, we have some evidence that in the
homology algebra $H_*Q_0\Sigma^{-2^i+6}P_{2^i-3}$ there are some
classes with nontrivial square. However, this does not imply that
the subalgebra generated by such classes is a polynomial algebra
inside this homology algebra as one still has to eliminate possible
truncations. Despite this disappointment, we are still able to show
existence of some classes in their homology algebra. However, we do
not know about the algebraic structure of the subalgebra that they
span. To be more precisely, notice that
$$\sigma_*^{2^i-6}\Sigma^{-2^i+6}a_{2^i-3}=a_{2^i-3}+\textrm{ other terms.}$$
This implies that if we choose, $I$ such that $\ex(I)\geqslant
2^i-3$ then $Q^I\Sigma^{-2^i+6}a_{2^i-3}\neq 0$. This comes easy
from the fact that
$$\sigma_*^{2^i-6}Q^I\Sigma^{-2^i+6}a_{2^i-3}=Q^Ia_{2^i-3}.$$
Hence, we may consider the subalgebra spanned by such elements.

\section{Relations to spherical classes homology of $Q_0S^0$}
The type of spherical classes in $H_*Q_0S^0$ are predicted by a
conjecture due to Curtis \cite[Thoerem 7.1]{7}. This predicts that
only the Hopf invairnat one elements
$\theta_i\in{_2\pi_{2^{i+1}-2}}Q_0S^0$ and the classical Hopf
invariant one elements in ${_2\pi_{2^i-1}Q_0S^0}$ map nontrivially
under the Hurewicz homomorphism
$$h:{_2\pi_*}Q_0S^0\to H_*Q_0S^0.$$
In fact the conjecture predicts that if $\alpha\in{_2\pi_*^S}$ has
Adams filtration at least $3$ then its stable adjoint viewed as an
element of ${_2\pi_*}Q_0S^0$ maps trivially under the Hurewicz
homomorphism. Notice that apart from the Hopf invariant one and
Kervaire invariant one elements we are left with the $\eta_i$ family
which we have dealt with in this paper.\\
Now, we can ask two related questions. First, notice that having
$\alpha\in{_2\pi_*}Q_0S^0$ we may consider to adjoint of alpha as
elements of ${_2\pi_{*-k}}Q_0S^{-k}$ under the suspension
isomorphism
$${_2\pi_{*-k}^S}S^{-k}\simeq{_2\pi_{*-k}}Q_0S^{-k}\to {_2\pi_*}Q_0S^0\simeq{_2\pi_*^S}$$
where $Q_0S^{-k}$ is the base point component of $\Omega^kQ_0S^0$.
We then may ask what is the least $k$ where the adjoint of $\alpha$
maps nontrivially under the Hurewics homomorphism
$${_2\pi_{*-k}}Q_0S^{-k}\to H_*Q_0S^{-k}.$$

Second, we may ask assuming that the Curtis's conjecture fails, how
we can calculate the Hurewicz image of those elements of which their
Adams filtration is at least $3$.\\

It seems to us that the answers to these questions are very much
related, and this work provides us with an example. This suggest
that an £$EHP$-approach is the right approach to deal with these
questions. We postpone more results and calculation on this to a
further work.

\section{Proof of Lemma 4.1}

Here we like to give a proof of Lemma 4.1. The following
observation, which is a corollary of the Freudenthal's suspension
theorem, will be used in the proof of lemma. 
\begin{mlmm}
Let $X_n^i$ denote a cell complex with bottom cell at dimension $n$
and top cell at dimension $i$. If $i<2n$ then $X_n^i$ admits at
least one desuspension, i.e.
$$X_n^i\simeq\Sigma Y_{n-1}^{i-1}.$$
\end{mlmm}

\begin{proof}
The proof is based on induction. If $i=n$, then $X_n^i$ is a wedge
of spheres and hence desuspends. Assume that the statement is true
for $X_n^i$, and we prove it for $X_n^{i+1}$. Let $f:S^i\to
X_n^i\simeq\Sigma Y_{n-1}^{i-1}$ denotes attaching map of an
$(i+1)$-cell. Observe that $f\in\pi_i\Sigma Y_{n-1}^{i-1}$ with
$i<2n$. According to the suspension theorem, $f$ desuspends to
$\pi_{i-1}Y_{n-1}^{i-1}$. The fact that $f$ desuspends implies that
the cofibre of $X_n^i\cup_fe^{i+1}$ also desuspends. Finally the
fact that $X_n^{i+1}$ is obtained by attaching some
$(i+1)$-dimensional cells through a map from a wedge of spheres to
$X_n^i$ shows that $X_n^{i+1}$ also admits a desuspension. This
completes the proof.
\end{proof}

Now we proceed with the proof of Lemma 8.\\

First, let $f:S^{2m+1}\to X$ be given such that $X$ has its bottom
cell at dimension $m+1$. Assume that $f$ is detected by $Sq^{m+1}$
on $x_{m+1}\in H_{m+1}X$. We like to show that the adjoint of $f$,
say $g:S^{2m}\to\Omega X$ is detected in homology by $hg=y_m^2\neq
0$ with $y_m\in H_m\Omega X$ such that $\sigma_*y_m=x_{m+1}$.\\

Notice that $f$ pulls back to the $(2m+1)$-skeleton of $X$, i.e. it
is in the image of
$$i_{\#}:\pi_{2m+1}X^{2m}\to \pi_{2m+1}X$$
where $i:X^{2m+1}\to X$ denotes the inclusion. We may apply Lemma 8
to $X^{2m+1}$ to observe that there exists a homotopy equivalence
$$X^{2m+1}\stackrel{\simeq}{\longrightarrow}\Sigma Y^{2m}$$
where $Y^{2m}$ has its bottom cell at dimension $m$ and top cell at
dimension $2m$. Now we may adjoint $f$ to obtain a mapping
$g:S^{2m}\to\Omega X$ where according to the above observation it
pulls back to a map $S^{2m}\to\Omega X^{2m+1}\simeq\Omega\Sigma
Y^{2m}$, i.e. we have the following commutative diagram
$$\xymatrix{ S^{2m}\ar[r]^-g\ar[rd]_-{g'} & \Omega X\\
                                    & \Omega X^{2m+1}\ar[u]_-{\Omega i}\ar[r]^-{\simeq}&\Omega\Sigma
                                    Y^{2m}.}$$
If we assume that $f$ is detected by $Sq^{m+1}$ on $x_{m+1}\in
H_{m+1}X$, this then also implies that the pull back of $f$ to
$X^{2m+1}$ is also detected by $Sq^{m+1}$ on $x_{m+1}=\Sigma y_{m}$
where $y_{2m}\in H_{m}Y^{2m}$. Lemma 6 then implies that the mapping
$$g':S^{2m}\to\Omega X^{2m+1}\simeq\Omega\Sigma Y^{2m}$$
is detected by homology, i.e. $hg'=y_m^2$ where we have used $y_m$
to denote the preimage of $y_m$ under the isomorphism $H_m\Omega
X^{2m+1}\to H_mY^{2m}$. The class $y_m$ has the
property that $\sigma_*y_m=x_{m+1}$.\\
To complete the proof, we need to show that $hg=(\Omega
i)_*y_m^2\neq 0$. This is straightforward once we consider the pair
$(\Omega X,\Omega X^{2m+1})$ and the following commutative diagram
with exact rows
$$\xymatrix{
\cdots\ar[r] &\pi_{2m+1}(\Omega X,\Omega X^{2m+1})\ar[r]^-\partial\ar[d]_-{h (\simeq)} & \pi_{2m}\Omega X^{2m+1}\ar[r]^-{(\Omega i)_{\#}}\ar[d]_-h & \pi_{2m}\Omega X\ar[r]\ar[d]_-h &\cdots\\
\cdots\ar[r] &  H_{2m+1}(\Omega X,\Omega X^{2m+1})\ar[r]^-\partial                     & H_{2m}\Omega X^{2m+1}\ar[r]^-{(\Omega i)_*}            & H_{2m}\Omega X\ar[r] &\cdots.\\
}$$ If we assume that $(\Omega i)_*hg'=(\Omega i)_*y_m^2=0$, then
$y_m^2$ pulls back to $H_{2m+1}(\Omega X,\Omega X^{2m+1})$. One may
use homotopy excision property to show that$$H_{2m+1}(\Omega
X,\Omega X^{2m+1})\simeq \pi_{2m+1}(\Omega X,\Omega X^{2m+1}),$$
i.e. $g'$ belongs to the image of $\partial:\pi_{2m+1}(\Omega X,\Omega
X^{2m+1})\to\pi_{2m}\Omega X^{2m+1}$. This then implies that
$(\Omega i)_{\#}g'=0$. However, we know that $0\neq g=(\Omega
i)_{\#} g$. This gives a contradiction to the assumption that
$(\Omega i)_*y_m^2=0$. Hence, $(\Omega i)_*y_m^2\neq 0$ and the
proof is
complete.\\
The proof of Lemma 7 in the other direction is done in a similar
way, i.e by a combination of Lemma 8 and Lemma 6, and we leave it to
the reader.

\section{One application and a conjecture}

Consider the case when $i=3$. In this case case we obtain spherical
classes $[\eta_3]_6\in H_6Q_0S^{-2}$ corresponding to $\eta_3$. A
quick observation is that this class dies under the homology
suspension $\sigma_*:H_*Q_0S^{-2}\to H_*Q_0S^{-1}$, and hence the
subalgebra of $H_*Q_0S^{-2}$ generated by $Q^I[\eta_3]_6$ belongs to
$\ker\sigma_*$. This is easy to see from the following fact.
\begin{lmm}
A spherical class $\xi_{-1}\in H_*Q_0S^{-1}$ survives under the
homology suspension $\sigma_*:H_*Q_0S^{-1}\to H_*Q_0S^0$.
\end{lmm}

\begin{proof}
Recall from \cite[Theorem 1.1]{3} that the homology suspension
$\sigma_*:QH_*Q_0S^{-1}\to PH_*Q_0S^0$ is an isomorphism where $Q$
is the indecomposable quotient module functor, and $P$ is the
primitive submodule functor. Moreover, the homology algebra
$H_*Q_0S^{-1}$ is an exterior given by
$$H_*Q_0S^{-1}\simeq E_{\Z/2}(\sigma^{-1}_*PH_*Q_0S^0).$$
Notice that a spherical class is primitive. This implies that a
spherical class in $H_*Q_0S^{-1}$ cannot be a decomposable, as if
this happens this it must be a square which is trivial in the
exterior algebra. Hence, a given spherical class $\xi_{-1}\in
H_*Q_0S^{-1}$ does not die under the suspension. This proves the
lemma.
\end{proof}

Now assuming that $\sigma_*[\eta_3]_6\neq 0$ would imply that
$\eta_i$ gives a spherical class in $H_*Q_0S^{-1}$ and hence to a
spherical class in $H_*Q_0S^0$. But this is a contradiction, as we
observed at the beginning of the paper that $\eta_i$ does not give
rise to a spherical class in $H_*Q_0S^0$. Hence,
$\sigma_*[\eta_3]_6=0$. In particular this detects a part of
$H_*Q_0S^{-2}$ which does not come from pull back of any class in
$H_*Q_0S^{-1}$. We note that the existing literature on the
calculation of $H_*Q_0S^{-2}$ has not detected this bit. This
motivates the following conjecture.
\begin{mconj}
the class $[\eta_i]_6\in H_6Q_0S^{-2^i+6}$ dies under the homology
suspension $\sigma_*:H_*Q_0S^{-2^i+6}\to H_{*+1}Q_0S^{-2^i+7}$.
Consequently, the subalgebra of $H_*Q_0S^{-2^i+6}$ generated by
$Q^I[\eta_i]_6$ belong to $\ker\sigma_*$.
\end{mconj}

\end{document}